\documentclass[]{siamart220329}

\usepackage{graphicx}
\graphicspath{{./fig/}}

\usepackage{subcaption}
\usepackage[shortlabels]{enumitem}

\usepackage{amsmath, amsfonts, amssymb}
\usepackage{tikz-cd}
\usepackage{mathtools}

\usepackage[mathscr]{euscript}
\usepackage{dsfont}

\usepackage{tikz}
\usepackage{tikz-cd}
\usetikzlibrary{calc}

\newtheorem{problem}[theorem]{Problem}
\newtheorem{example}[theorem]{Example}
\newtheorem{remark}[theorem]{Remark}

\newcommand{\define}[1]{\emph{#1}}

\newcounter{refer}
\newcommand{\labellocal}[1]{\label{local:\therefer:#1}}
\newcommand{\reflocal}[1]{\ref{local:\therefer:#1}}
\newcommand{\Creflocal}[1]{\Cref{local:\therefer:#1}}

\newcommand{\Span}{\operatorname{Span}}
\newcommand{\Ima}{\operatorname{Im}}
\newcommand{\Ker}{\operatorname{Ker}}

\newcommand{\im}{\operatorname{Im}}

\newcommand\restr[2]{{
		\left.\kern-\nulldelimiterspace 
		#1 
		\vphantom{\big|} 
		\right|_{#2} 
	}}
	
\newcommand{\RR}{\mathbb{R}}

\newcommand{\NN}{\mathbb{N}}
\newcommand{\ZZ}{\mathbb{Z}}
\newcommand{\CC}{\mathbb{C}}

\newcommand{\CCge}[1]{\CC_{\geq#1}}
\newcommand{\CCl}[1]{\CC_{<#1}}

\newcommand{\RRg}[1]{\RR_{>#1}}
\newcommand{\RRge}[1]{\RR_{\geq#1}}

\newcommand{\RRle}[1]{\RR_{\leq#1}}

\newcommand{\ZZge}[1]{\ZZ_{\geq#1}}

\newcommand{\local}{\mathrm{loc}}

\newcommand{\hooktwoheadrightarrow}{%
  \hookrightarrow\mathrel{\mspace{-15mu}}\rightarrow
}

\newcommand{\spec}{\operatorname{Spec}}
\newcommand{\domain}{\operatorname{Dom}}

\newcommand{\diag}{\operatorname{Diag}}
\newcommand{\gradient}{\nabla}

\newcommand{\system}{\operatorname{Sys}}

\newcommand{\sysclose}{\operatorname{Close}}

\newcommand{\vertiii}[1]{{\left\vert\kern-0.25ex\left\vert\kern-0.25ex\left\vert #1 
	\right\vert\kern-0.25ex\right\vert\kern-0.25ex\right\vert}}

\newenvironment{talign*}
	{\let\displaystyle\textstyle\csname align*\endcsname}
	{\endalign}

\title{Exponential stabilization of infinite-dimensional systems by finite-dimensional controllers\thanks{Submitted to the editors on September 04, 2023
\funding{This research was supported by the Natural Sciences and Engineering Research Council of Canada (NSERC).}}
}

\author{Tian Xia\thanks{The Edward S. Rogers Sr. Department of Electrical \& Computer Engineering, University of Toronto, Toronto, ON, M5S 3G4, Canada
    (\email{t.xia@mail.utoronto.ca}; \email{luca.scardovi@utoronto.ca}).}
\and Giacomo Casadei\thanks{Laboratoire Amp\`{e}re Dpt. EEA of the \'{E}cole Centrale de Lyon, Universit\'{e} de Lyon, 69134 \'{E}cully, France.}
\and Francesco Ferrante\thanks{University of Perugia, Department of Engineering, Perugia, Italy.}
\and Luca Scardovi\footnotemark[2]
}

\begin{document}

\maketitle

\newif\ifshowpages
\showpagestrue
\ifshowpages
	\pagestyle{plain}
	\thispagestyle{plain}
\else
	\pagestyle{empty}
	\thispagestyle{empty}
\fi

\begin{abstract}
	This paper studies the feedback stabilization of abstract Cauchy problems with unbounded output operators by finite-dimensional controllers. 
	Both necessary conditions and sufficient conditions for feedback stabilizability are presented.
	The proof of closed-loop stability is based on a novel input-output gain introduced in this paper.
	For systems satisfying a property we call quasi-finite, an equivalent characterization of feedback stabilizability is obtained.
	Quasi-finiteness is verified for classes of parabolic and hyperbolic equations. 
\end{abstract}



\section{Introduction}
\subsection{Background}
The control of infinite-dimensional systems has been an active area of research for many decades. Indeed, infinite-dimensional systems occur in many applications ranging from engineering, physics, and biology.
In the linear time-invariant setting, many results developed for finite-dimensional systems can naturally be extended  to the infinite-dimensional setting; see \cite{curtain95} for a complete monograph on this subject. In particular, whenever the input and output operators are bounded, stabilizable and detectable systems can always be feedback stabilized by an infinite-dimensional controller; see \cite[Chapter 5]{curtain95}. However, 
this approach has two main limitations. The first limitation is that infinite-dimensional controllers are hard to implement. The second limitation is that assuming boundedness of the input and output operators dramatically limits the range of applications that can be addressed. A possible workaround to this second limitation lies in the use of the so-called backstepping \cite{krstic08}, which has been a very successful control design methodology for infinite dimensional systems. However, backstepping too leads to infinite-dimensional controllers. 

The design of finite-dimensional controllers for infinite-dimensional systems has been considered in the 80's \cite{curtain82,balas88}, but under fairly restrictive assumptions on the input and output operators.
Specifically, \cite{balas88} analyzes systems with bounded input/output operators, and \cite{curtain82} imposes assumptions to ensure a modal separation. By relying on the ideas in  \cite{fattorini1971exact}, where the so-called ``Reduction to Moment Problems'' is introduced in the context of controllability analysis of systems modeled via linear parabolic partial differential equations (\emph{PDE}s), more recent results lift these restrictions in various directions see, e.g., \cite{prieur19,mironchenko2020local}. In particular, the approach pursued in this line of works can be roughly illustrated as follows. By relying on a spectral decomposition of the system dynamic operator, the system is rewritten as the interconnection of a finite-dimensional unstable system and an infinite-dimensional (exponentially) stable one. Then, a finite-dimensional controller is designed to stabilize the finite-dimensional unstable component of the plant while preserving closed-loop stability of the actual (infinite-dimensional) system. A review of this approach, as well its application to systems governed by hyperbolic PDEs is presented in 
\cite[Section 4]{russell1978controllability}. Similar ideas have been later used also in the context of semilinear systems in \cite{coron2004global,coron2006global}. 

Due to its effectiveness, the idea of relying on a spectral decomposition of the system operator has seen an increasing attention in the community in the last few years. In this context, researchers have proposed several extensions of the seminal work by Fattorini covering different problems ranging from output feedback stabilization to set-point regulation for linear, semilinear, and nonlinear infinite-dimensional systems. A non-exhaustive list of references follows. 
 
In \cite{katz20}, the spectral decomposition technique is used to derive constructive conditions for observer-based control design for $1$-D parabolic equations. A similar approach is used in \cite{lhachemi22}, where finite-dimensional observer-based control design of reaction-diffusion equations with Dirichlet/Neumann boundary measurement is considered. 
This approach has been later extended to account for delays in \cite{katz2020boundary,lhachemi2022boundary,lhachemi2023output}. 
Moreover, extensions to the problem of set-point tracking as well as to specific classes of semilinear/nonlinear systems can be found in \cite{lhachemi2021finite,lhachemi2022proportional,lhachemi2022local}. 
In \cite{dus21}, state feedback control design for linear and saturated transport equations with boundary and in-domain couplings is studied. Finite-dimensional output feedback input-output stabilization of a reaction diffusion system in the presence of in-domain disturbances is tackled in \cite{shreim2022input}.
 
\subsection{Contributions and outline of the paper}

A limiting aspect of the results mentioned so far is that they apply to specific classes of infinite-dimensional systems. Our paper overcomes this limitation by focusing on the problem of stabilization by finite-dimensional controllers in a more general framework. Within this setting, necessary and sufficient conditions for the solvability of this problem are provided. 

The sufficient conditions we provide are based on the constructions of \cite{lhachemi22}:
the system is decomposed into its stable and unstable modes, and an observer is built for the unstable mode and a sufficient large subspace of the stable mode.
As opposed to \cite{katz20,lhachemi22}, where closed-loop asymptotic stability is ensured by a Lyapunov argument, we make use a specifically tailored small-gain analysis, which, in our opinion, is better suited for the analysis of composite systems.

To handle systems with unbounded input/output operators, we find it convenient to step outside of the abstract Cauchy problem framework and consider input-state-output relations (which we describe in \Cref{sect:inputoutput_relation}). Necessary conditions for stabilization by finite-dimensional controller are presented in \Cref{sect:necessary}, while sufficient conditions are provided in \Cref{sect:sufficient}.
The applicability of the proposed results is illustrated on both parabolic and hyperbolic systems in \Cref{sect:applications}.

\subsection{Notations}

We denote by $\CCge{a}$ the subset of $\CC$ with real part greater or equal to $a$.
The notations $\CCl{a}$, $\RRg{a}$, $\RRle{a}$ are analogously defined.
The space of $k$-times continuously differentiable $Y$-valued functions on $X$ is denoted by $C^{k}_{\local}(X, Y)$.
The symbol $H^k(X)$ denotes the $L^2$-based Sobolev space of order $k$.
$H^k_0(X)$ denotes the closure of $C^{\infty}_{c}(X)$ in $H^k(X)$, where $C^{\infty}_{c}(X)$ denotes smooth compactly supported functions on $X$ (note that $C^{\infty}_{c}([0,1))$ must vanish near $1$, but not $0$).
Given real-valued functions $f$ and $g$ on $X$, the inequality $f(x) \lesssim g(x)$ uniform in $x \in S \subseteq X$ means that there exists $a > 0$ such that $f(x) \leq a g(x)$ for all $x \in S$.

\section{Problem statement}

Our objective is to study the exponential stabilization of infinite-dimensional systems by finite-dimensional controllers.
We start by giving an intuitive, but technically imprecise statement of the problem.

Consider the abstract Cauchy problem
\begin{equation} \begin{aligned} \label{eq:system}
	\dot{x} &= Ax + Bu, \\
	y &= Cx,
\end{aligned} \end{equation}
where $A$ generates a strongly continuous semigroup on the state space $X = \domain(e^A)$, $B : U \rightarrow X$ is bounded, and $C : \domain(A) \rightarrow Y$ is bounded (note that $C$ may be unbounded as an operator $X \rightarrow Y$).
Denote \Cref{eq:system} by $\system(A, B, C)$.
$\system(A, B, C)$ is finite-dimensional if $\dim X < \infty$.
In this case, $\domain(A)$ must be equal to $X$ and $C$ must be bounded on $X$.
$\system(A, B, C)$ is called exponentially stable if $\|e^{At}\| \lesssim e^{-\beta t}$ uniform in $t \geq 0$ for some $\beta > 0$.

\begin{remark} \label{rmk:system_assumption}
	Whenever $\system(E, F, G)$ is written, we assume that $E$ generates a strongly continuous semigroup on a Banach space $W$, $F$ is a bounded operator from a Banach space to $W$, and $G$ is a bounded operator from $\domain(E)$ to a Banach space.
\end{remark}

\begin{problem} \label{prob:control}
	Given a system $\system(A, B, C)$, design a finite-dimensional controller $\system(E, F, G)$ such that their closed-loop composition\footnote{The notation $\partial$ is used for both partial derivatives and the total derivative when the function is single-variable.}
	\begin{align} \label{eq:closedloop}
		\partial_t \begin{bmatrix} x \\ w \end{bmatrix} &= \begin{bmatrix} A & BG \\ FC & E \end{bmatrix} \begin{bmatrix} x \\ w \end{bmatrix}
	\end{align}
	is exponentially stable.
\end{problem}

The problem statement is not entirely precise because the composition of abstract Cauchy problems $\system(\cdot, \cdot, \cdot)$ may not be an abstract Cauchy problem (due to the unbounded output operators).
The definitions of the closed-loop system and its exponential stability are therefore ambiguous.
These ambiguities will be resolved by extending these notions to a larger class of systems containing abstract Cauchy problems and closed under compositions. 
We will consider the class of systems described by input-state-output relations, formally defined in the next section.
This extension will also be essential for our stability analysis in \Cref{sect:sufficient}.

\subsection{Input-state-output relations} \label{sect:inputoutput_relation}

Let $U$, $X$, $Y$ be Banach spaces, any subset of $C^{0}_{\local}(\RRge{0}, U) \times C^{0}_{\local}(\RRge{0}, X) \times C^{0}_{\local}(\RRge{0}, Y)$ is called a \define{$(U, X, Y)$-relation}. 
Relations are used to represent possible input-state-output behaviours of a system.
Any set of equations together with input/state/output specifications defines an input-state-output relation.
For example, $\dot{x} = u + \dot{u} + \ddot{u} ,\, w + \dot{u} = 0$ with input $u$, state $x$, and output $y = (w, \dot{w})$ defines a relation consisting of $(u, x, y)$ such that all terms in the equations are defined, continuous, and satisfy the equations above.
In this case, the regularities of the functions are $u \in C^{2}_{\local}(\RRge{0}, U)$, $x \in C^{1}_{\local}(\RRge{0}, X)$ and $y \in C^{0}_{\local}(\RRge{0}, W \times W)$. Similarly, the system $\system(A, B, C)$ defined by \Cref{eq:system} consists of $(u, x, y)$ solving \Cref{eq:system} with $u \in C^{0}_{\local}(\RRge{0}, U)$, $x \in C^{0}_{\local}(\RRge{0}, \domain(A)) \cap C^{1}_{\local}(\RRge{0}, X)$, $y \in C^{0}_{\local}(\RRge{0}, Y)$.
In other words, $\system(A, B, C)$ is the relation consisting of classical solutions of \Cref{eq:system}. From here on we will use the terms relation and system interchangeably.

A system $S$ is called \define{finite-dimensional} if $\dim X < \infty$. Compositions of systems are defined naturally.
Let each $S_j$ be a $(U_j, X_j, Y_j)$-relation.
If $Y_1 = U_2$, we denote by $S_2 \circ S_1$ the serial composition.
If $U_1 = U_2$ and $Y_1 = Y_2$, we denote by $S_1 + S_2$ the parallel composition, whose output is the sum of outputs of $S_1$ and $S_2$.
Note that both $S_2 \circ S_1$ and $S_1 + S_2$ have state space $X_1 \times X_2$, equipped with the product norm.
If $Y_1 = U_1$, we denote by $\sysclose(S_1)$ the closed-loop composition of $S_1$ under the feedback $u_1 = y_1$, as a relation with no input, state $x_1$, and no output.
Two systems $S_1$, $S_2$ are isomorphic (denoted $S_1 \cong S_2$) if $U_1 = U_2$, $Y_1 = Y_2$ (as Banach spaces), and there exists a topological isomorphism $X_1 \hooktwoheadrightarrow X_2$ such that the induced bijection from $C^{0}_{\local}(\RRge{0}, U_1) \times C^{0}_{\local}(\RRge{0}, X_1) \times C^{0}_{\local}(\RRge{0}, Y_1)$ to $C^{0}_{\local}(\RRge{0}, U_2) \times C^{0}_{\local}(\RRge{0}, X_2) \times C^{0}_{\local}(\RRge{0}, Y_2)$ sends $S_1$ surjectively into $S_2$.
Thus, $S_1 + S_2$ and $S_2 + S_1$ are isomorphic even though they have distinct state spaces $X_1 \times X_2$ and $X_2 \times X_1$.
It is worth noting that serial and parallel compositions preserve isomorphism classes of systems.  

\begin{definition}
	System $S$ is called \define{$\beta$-exponentially stable} if $\|x(t)\| \lesssim e^{-\beta t}\|x(0)\|$ uniform in $(0, x, y) \in S$ and $t \geq 0$;
	$S$ is \define{exponentially stable} if it is $\beta$-exponentially stable for some $\beta > 0$; 
\end{definition}

The representation of a system (relation) as $\system(A, B, C)$ is unique if it exists.
So there is no ambiguity in defining properties of systems using the operators $A, B, C$.
\begin{proposition}
	If a $(U, X, Y)$-relation $S$ 
	is equal to $\system(A, B, C)$, then $A$, $B$, $C$ are uniquely determined.
\end{proposition}
\begin{proof}
	Elements of $S$ with zero input determine $e^{At}$ on $\domain(A)$, hence $e^{At}$ is unique by density.
	Thus, $A$ is uniquely determined. Uniqueness of $B$ and $C$ follows trivially.
\end{proof}

The class of systems $\system(\cdot, \cdot, \cdot)$ is in general not closed under compositions, due to the unboundedness of the output operator.
The following proposition shows that the closed-loop composition in \Cref{prob:control} remains a system of the form $\system(\cdot, \cdot, \cdot)$.

\begin{proposition} \label{pro:closedloop_semigroup} \stepcounter{refer}
	Given a $(U, X, Y)$-relation $S_1 = \system(A, B, C)$ and a $(Y, W, Z)$-relation $S_2 = \system(E, F, G)$.
	If $E$ is bounded, then there is an isomorphism of input-state-output relations
	\begin{align} \labellocal{eq:1}
		S_2 \circ S_1 &\cong \system\left( \begin{bmatrix} A & 0 \\ \lambda FC_\lambda - EFC_\lambda & E \end{bmatrix},\, \begin{bmatrix} B \\ F C_{\lambda} B \end{bmatrix},\, \begin{bmatrix} -GFC_\lambda & G \end{bmatrix} \right),
	\end{align}
	where $\lambda$ is any fixed number in $\CC \setminus \spec(A)$ and $C_\lambda$ is the bounded operator $C(\lambda - A)^{-1} : X \rightarrow Y$.
	Furthermore, if $U = Z$, then the closed-loop composition $\sysclose(S_2 \circ S_1)$ is isomorphic to
	\begin{align} \labellocal{eq:2}
		\system\left( \begin{bmatrix} A & 0 \\ 0 & E \end{bmatrix} + \begin{bmatrix} -BGFC_{\lambda} & BG \\ \lambda F C_\lambda - EFC_\lambda - FC_{\lambda}BGFC_{\lambda} & FC_{\lambda}BG \end{bmatrix}, 0, 0 \right).
	\end{align}
\end{proposition}
\begin{proof}
	The system $S_1$ is defined by the equations 
	\begin{align*}
\dot{x}_1 &= Ax_1 + Bu_1\\
 y_1 &= Cx_1
	\end{align*}	
	By combining the two equations above, $y_1$ can be rewritten as
	\begin{align*}
		y_1 &= C (\lambda - A)^{-1} (\lambda x_1 - Ax_1) = \lambda C_\lambda x_1 - C_\lambda \dot{x}_1 + C_\lambda B u_1.
	\end{align*}
	Therefore, the serial composition $S_2 \circ S_1$ is the system defined by
	\begin{align*}
		\dot{x}_1 &= A x_1 + B u_1, \\
		\dot{x}_2 &= E x_2 + \lambda F C_\lambda x_1 - F C_\lambda \dot{x}_1 + F C_\lambda B u_1, \\
		y_2 &= G x_2, 
	\end{align*}
	with input $u_1$, state $(x_1, x_2)$, and output $y_2$.
	Letting $x_3 = x_2 + FC_{\lambda}x_1$, then
	\begin{align*}
		\dot{x}_1 &= A x_1 + B u_1, \\
		\dot{x}_3 &= E x_3 - EFC_{\lambda}x_1 + \lambda F C_\lambda x_1 + F C_\lambda B u_1, \\
		y_2 &= G x_3 - GFC_{\lambda} x_1. 
	\end{align*}
	This implies \Creflocal{eq:1}.
	\Creflocal{eq:2} follows from \Creflocal{eq:1} because the input/output operators of $S_2 \circ S_1$ are bounded.
\end{proof}

\section{Necessary conditions and preliminary analysis} \label{sect:necessary}

In this section, we identify and investigate some necessary conditions for the solvability of \Cref{prob:control}.
It is easy to see that stabilizability and detectability are required, as in the finite-dimensional case.
These notions are defined below to avoid ambiguity.

\begin{definition}
	A $(U, X, Y)$-relation $S = \system(A, B, C)$ is \define{stabilizable} if the set\footnote{This is the set of initial conditions admitting an input $u$ driving the state $x$ to zero asymptotically.}
	\begin{align*}
		\{x(0) \mid (u, x, y) \in S \text{ and } \lim_{t \rightarrow \infty} \|x(t)\| = 0\}
	\end{align*}
	is dense in $X$.
	It is \define{detectable} if every $(0, x, 0) \in S$ satisfies $\lim_{t \rightarrow \infty} \|x(t)\| = 0$.
\end{definition}

\begin{remark}
	When $\system(A, B, C)$ is finite-dimensional, stabilizability is classically defined as the existence of $u$ for every initial condition $x(0) \in X$ such that $x$ converges to zero.
	However, a general infinite-dimensional system does not possess any $(u, x, y)$ initialized at $x(0) \not\in \domain(A)$.
	Our definition of stabilizability is therefore a natural generalization of the classical definition to the infinite-dimensional setting.
\end{remark}

A less obvious necessary condition, which will come into play later, is that $S$ must have finite unstable part.

\begin{definition} \label{def:finite_unstable} \stepcounter{refer}
	A $(U, X, Y)$-relation $S = \system(A, B, C)$ has \define{finite unstable part} if there is a decomposition $S \cong S_u + S_s$ for a finite-dimensional $S_u = \system(A_u, B_u, C_u)$ and an exponentially stable $S_s = \system(A_s, B_s, C_s)$.
\end{definition}

\subsection{Necessary conditions}

This section proves the following result.

\begin{theorem} \label{pro:necessary}
	Let $S = \system(A, B, C)$.
	If there exists finite-dimensional $T = \system(E, F, G)$ such that $\sysclose(T \circ S)$ is exponentially stable, then $S$ is stabilizable, detectable, and has finite unstable part.
\end{theorem}

$S$ having finite unstable part is the only conclusion needing justification.
By \Cref{pro:closedloop_semigroup}, exponential stability of $\sysclose(T \circ S)$ implies that a finite-rank perturbation of the operator $\diag(A, 0)$ is exponentially stable.
Following the result from \cite{curtain95} recalled below, this implies that the operators $\diag(A, 0)$ and $A$ both have finite unstable parts.

\begin{definition} \label{def:finite_unstable_2} \stepcounter{refer}
	A semigroup generator $A$ has \define{finite unstable part} if $\domain(e^A)$ decomposes into a direct sum $X_u \oplus X_s$ of closed $A$-invariant\footnote{A subspace of $V \subseteq X$ is called $A$-invariant if every transition operator $e^{At}$ maps $V$ into $V$.} subspaces such that $\dim X_u < \infty$ and the restriction of $e^{At}$ to $X_s$ is an exponentially decaying\footnote{There exists $\epsilon > 0$ such that $\|e^{At}\| \lesssim e^{-\epsilon t}$ uniform in $t \geq 0$.} semigroup.
\end{definition}

\begin{lemma}[{\cite[Lemma 5.2.4, Thm. 5.2.6]{curtain95}}] \label{lem:finite_unstable} \stepcounter{refer}
	Assume that $A$ is the sum of the generator of an exponentially decaying semigroup  and a finite rank operator.
	Then $A$ has finite unstable part.
\end{lemma}

The notions of finite unstable part defined for system $S = \system(A, B, C)$ and operator $A$ are in fact equivalent, as proved in the following result.
Therefore, $S$ must have finite unstable part.

\begin{lemma} \label{pro:finite_unstable} \stepcounter{refer}
	Let $S$ be the $(U, X, Y)$-relation $\system(A, B, C)$.
	The decompositions $S \cong S_u + S_s$ as in \Cref{def:finite_unstable} correspond one-to-one to the decompositions of $X = X_u \oplus X_s$ as in \Cref{def:finite_unstable_2}.
\end{lemma}
\begin{proof}
	$S \cong S_u + S_s$ implies that $S$ is equal, in some coordinates on $X$, to
	\begin{align*}
		\system(A_u, B_u, C_u) + \system(A_s, B_s, C_s) 
		&= \system\left( \begin{bmatrix} A_u & 0 \\ 0 & A_s \end{bmatrix},\, \begin{bmatrix} B_u \\ B_s \end{bmatrix},\, \begin{bmatrix} C_u & C_s \end{bmatrix} \right).
	\end{align*}
	This gives rise to a decomposition $X = X_u \oplus X_s$.

	Conversely, given $X = X_u \oplus X_s$, the transition operators $e^{At}$ writes in these coordinates as
	\begin{align*}
		e^{At}=\begin{bmatrix} T_u(t) & 0 \\ 0 & T_s(t) \end{bmatrix} : \begin{matrix} X_u \\ \oplus \\ X_s \end{matrix} 
		\rightarrow \begin{matrix} X_u \\ \oplus \\ X_s \end{matrix}.
	\end{align*}
	The restrictions $T_u$ and $T_s$ remain strongly continuous, and their generators $A_u$ and $A_s$ coincide with the restrictions of $A$ to $X_u$ and $X_s$ respectively.
	Furthermore, $\domain(A) = \domain(A_u) \oplus \domain(A_s)$.
	This implies that $S \cong \system\left( \begin{bmatrix} A_u & 0 \\ 0 & A_s \end{bmatrix},\, \begin{bmatrix} B_u \\ B_s \end{bmatrix},\, \begin{bmatrix} C_u & C_s \end{bmatrix} \right)$
	where $B_u$, $B_s$ are composition of $B$ with projections, and $C_u$, $C_s$ are restrictions of $C$.
\end{proof}

\Cref{pro:necessary} follows from the application of \Cref{lem:finite_unstable} and \Cref{pro:finite_unstable}.

\subsection{Characterization of systems with finite unstable parts}
Our later analysis will rely on some equivalent characterizations of the necessary conditions of the previous section.

For systems with finite unstable parts, stabilizability and detectability can be verified by analyzing only the ``unstable part''.
This simple result is stated below.

\begin{proposition} \label{lem:stabilizability} \stepcounter{refer}
	Assume system $S$ decomposes into $S_u + S_s$ as in \Cref{def:finite_unstable}. Then the following statements are equivalent:
	\begin{enumerate}[(a)] \itemsep 0mm
		\item\labellocal{item:a} $S$ is stabilizable (resp. detectable).
		\item\labellocal{item:b} $S_u$ is stabilizable (resp. detectable).
	\end{enumerate}
\end{proposition}
\begin{proof}
	``\reflocal{item:b} $\Rightarrow$ \reflocal{item:a}''
	Assume the finite-dimensional $S_u = \system(A_u, B_u, C_u)$ is stabilizable, then there exists $K$ such that $A_u + B_u K$ is Hurwitz. 
	Under the feedback $u = K x_u$, the system $S$ becomes
	\begin{align*}
		\partial_t \begin{bmatrix} x_u \\ x_s \end{bmatrix} &= \begin{bmatrix} A_u + B_u K & 0 \\ B_s K & A_s \end{bmatrix} \begin{bmatrix} x_u \\ x_s \end{bmatrix},
	\end{align*}
	which is exponentially stable.
	The detectability property follows by a dual argument.
\end{proof}

The decomposition $X = X_u \oplus X_s$ in \Cref{pro:finite_unstable} is typically not unique, but there is a canonical decomposition $X = X_{\bar{u}} \oplus X_{\bar{s}}$ where $X_{\bar{u}}$ is minimum among all subspaces containing the unstable modes.

\begin{proposition} \label{pro:finite_unstable2} \stepcounter{refer}
	Let $S$ be the $(U, X, Y)$-relation $\system(A, B, C)$.
	If $S$ has finite unstable part, then: 
	\begin{enumerate}[(a)]
		\item\labellocal{item:a} $\spec(A) = Z_{\bar{u}} \cup Z_{\bar{s}}$ where $Z_{\bar{u}} \subseteq \CCge{0}$ is finite and $Z_{\bar{s}} \subseteq \CCl{\omega}$ for some $\omega < 0$.
	\end{enumerate}

	Let $\Gamma : S^1 \rightarrow \CCge{\omega} \setminus \spec(A)$ be any simple closed curve positively encircling $Z_{\bar{u}}$, and consider the projection operator $P_{\bar{u}}$ defined by the functional calculus formula (see \cite[Lemma 2.5.7]{curtain95} for details)
	\begin{align*}
		P_{\bar{u}} = \frac{1}{2\pi i} \int_{\Gamma} (z - A)^{-1} dz.
	\end{align*}
	Set $X_{\bar{u}} = \Ima P_{\bar{u}}$ and $X_{\bar{s}} = \Ima P_{\bar{s}} = \Ima (I - P_{\bar{u}})$, then
	\begin{enumerate}[(a),resume]
		\item\labellocal{item:0} The restriction of $A$ to $A_{\bar{u}}$ and $A_{\bar{s}}$ satisfies $\spec(A_{\bar{u}}) \subseteq \CCge{0}$ and $\spec(A_{\bar{s}}) \subseteq \CCl{0}$.
		\item\labellocal{item:b} For any decomposition $X = X_1 \oplus X_2$ into closed $A$-invariant subspaces, $X_j = (X_j \cap X_{\bar{u}}) \oplus (X_j \cap X_{\bar{s}})$ for $j \in \{1, 2\}$.
		\item\labellocal{item:c} If the restriction of $e^{At}$ to $X_j$ in \reflocal{item:b} is exponentially decaying, then $X_j \cap X_{\bar{u}} = 0$.
		\item\labellocal{item:d} If $X = X_u \oplus X_s$ as in \Cref{pro:finite_unstable}, then there is a closed $A$-invariant subspace $X_r$ such that $X_u = X_{\bar{u}} \oplus X_r$ and $X_{\bar{s}} = X_s \oplus X_r$.
		\item\labellocal{item:e} $\dim X_{\bar{u}} < \infty$ and the restriction of $e^{At}$ to $X_{\bar{s}}$ is exponentially decaying.
	\end{enumerate}
\end{proposition}
\begin{proof}
	``Proof of \reflocal{item:a}''
	Since $S$ has finite unstable part, $X = X_u \oplus X_s$ as in \Cref{def:finite_unstable_2}.
	Let $A_u$ and $A_s$ be the restrictions of $A$ to $X_u$ and $X_s$, then $\spec(A) = \spec(A_u) \cup \spec(A_s)$.
	Setting $Z_{\bar{u}} = \spec(A_u) \cap \CCge{0}$ and $Z_{\bar{s}} = (\spec(A_u) \setminus \CCge{0}) \cup \spec(A_s)$ completes the proof.
	Note that $\spec(A_s) \subseteq \CCl{\omega}$ because $e^{A_s t}$ is exponentially decaying.
	
	``Proof of \reflocal{item:0}''
	see \cite[Lemma 2.5.7]{curtain95}.
	
	``Proof of \reflocal{item:b}''
	By the definition, $P_{\bar{u}}$ and $P_{\bar{s}}$ commute with any operator commuting with the resolvents of $A$.
	In particular, they commute with the canonical projection on $X_1 \oplus X_2$, denoted by $P_1$ and $P_2$.
	This implies that $P_1 P_{\bar{u}}$, $P_1 P_{\bar{s}}$, $P_2 P_{\bar{u}}$, $P_2 P_{\bar{s}}$ are projections summing to identity.
	By commutativity, $\Ima P_1 P_{\bar{u}} \subseteq X_1 \cap X_{\bar{u}}$.
	Since the projections sum to identity, this inclusion is an equality.
	The same holds for the other projections.
	The conclusion follows from $\Ima P_j = \Ima P_j P_{\bar{u}} + \Ima P_j P_{\bar{s}}$.
	
	``Proof of \reflocal{item:c}''
	The restriction of $A$ to $X_j \cap X_{\bar{u}}$ is also the restriction of $A_j$ or $A_{\bar{u}}$ to $X_j \cap X_{\bar{u}}$.
	Since $\spec(A_{\bar{u}}) \subseteq \CCge{0}$ and $\spec(A_j) \subseteq \CCl{0}$ ($A_j$ generates an exponentially decaying semigroup), the restriction of $A$ to $X_j \cap X_{\bar{u}}$ is a semigroup generator with empty spectrum.
	Therefore, $X_j \cap X_{\bar{u}} = 0$.
	
	``Proof of \reflocal{item:d}''
	By \reflocal{item:b}, $X$ is the direct sum of $X_u \cap X_{\bar{u}}$, $X_u \cap X_{\bar{s}}$, $X_s \cap X_{\bar{u}}$, and $X_s \cap X_{\bar{s}}$.
	By \reflocal{item:c}, $X_s \cap X_{\bar{u}} = 0$.
	The conclusion holds with $X_r = X_u \cap X_{\bar{s}}$.
	
	``Proof of \reflocal{item:e}''
	follows from \reflocal{item:d} and \reflocal{item:0}.
\end{proof}

%

\section{Sufficient conditions} \label{sect:sufficient}
It turns out that, under some mild assumptions on the operator $A$, the conditions in \Cref{pro:necessary} are also sufficient for the solvability of \Cref{prob:control}.     

\begin{definition} \label{def:schauderbasis}
	A \define{Schauder subspace-basis of $X$} is a countable collection $\{X_j \mid j \in \NN\}$ of finite-dimensional subspaces such that every $x \in X$ can be written as $x=\sum_{j \in \NN} x_j$ for a unique sequence $x_j \in X_j$.
	We say that the subspace-basis $\{X_j\}_{j \in \NN}$ is \define{adapted to $A$} if every $X_j$ is $A$-invariant.
\end{definition}

\begin{theorem} \label{thm:control_schauderbasis}
	Let $S = \system(A, B, C)$ be a $(U, X, Y)$-relation such that $X$ has a Schauder subspace-basis adapted to $A$.
	\Cref{prob:control} is solvable for $S$ if and only if $S$ is stabilizable, detectable, and has finite unstable part.
\end{theorem}

\begin{remark}
It is worth noting that operators $A$ with adapted Schauder subspace-basis include
\begin{itemize}
	\item self-adjoint operator with compact resolvents on a separable Hilbert space;
	\item spectral operators \cite{dunford71} with compact resolvents;
	\item Riesz-spectral operators \cite[Sect. 2.3]{curtain95}.
\end{itemize}
\end{remark}

The main contribution of the remaining part of this section is to prove \Cref{thm:control_schauderbasis} by constructing a finite-dimensional controller stabilizing $S$. 
The design strategy is outlined below.

Given a stabilizable, detectable system $S$ with finite unstable part, write $S \cong S_u + S_s$ as in \Cref{def:finite_unstable}.
Because $S_u = \system(A_u, B_u, C_u)$ is a stabilizable and detectable finite-dimensional system, there exists an observer-based controller $T = \system(A_u + LC_u + B_uK, -L, K)$ such that $\sysclose(T \circ S_u)$ is exponentially stable. However, we seek exponential stability of the full feedback system $\sysclose(T \circ S)$ shown in \Cref{fig:closedloop_diagram}.
\begin{figure}[h]
\centering \begin{tikzpicture}
	\tikzstyle{sys}=[rectangle, draw, minimum width=10mm, minimum height=8mm];
	\node[sys] at (0,0) (S_u) {$S_u$};
	\node[sys] at (0,-1.2) (S_s) {$S_s$};
	\node[sys] at (0,1.2) (T) {$T$};
	\node[circle, draw, scale=0.7] at (1.5, 0) (out) {$+$};
	\coordinate (in) at (-1.5, 0);
	
	\draw[->] (S_u) -- (out);
	\draw[->] (S_s) -| node[above, pos=0.25] {$y_s$} (out);
	\draw[->] (out) |- node[above, pos=0.75] {$y$} (T);
	
	\draw[->] (T) -| (in) -- (S_u);
	\draw[->] (in) |- node[above, pos=0.75] {$u$} (S_s);
	
	\draw[dotted, thick] (-2, -0.6) -- (2, -0.6) --node[right] {$R$} (2, 1.8) -- (-2, 1.8) -- cycle;
\end{tikzpicture} 
\caption{Closed-loop system $\sysclose(T \circ S)$} \label{fig:closedloop_diagram}
\end{figure}

Let $R$ be the system labelled in \Cref{fig:closedloop_diagram} with input $y_s$ and output $u$, then $\sysclose(T \circ S) \cong \sysclose(R \circ S_s)$.
Because $R$ and $S_s$ are exponentially stable, we expect $\sysclose(R \circ S_s)$ to remain exponentially stable if the IO gains of $R$ and $S_s$ are sufficiently small.
Furthermore, the freedom of choosing the decomposition $S \cong S_u + S_s$ allows us to, in some situations, reduce the IO gain of $S_s$ arbitrarily.
This can be used to prove exponential stability of $\sysclose(T \circ S)$.

To formalize the idea above, we will need an adapted small-gain theorem presented in \Cref{sect:small_gain}.
This small-gain result (\Cref{pro:smallgain}) will be used in \Cref{sect:control_design} to prove a generalization of \Cref{thm:control_schauderbasis} based on a weaker assumption called quasi-finiteness.
In \Cref{sect:quasifinite}, we will recover \Cref{thm:control_schauderbasis} by showing that having a Schauder subspace-basis implies quasi-finiteness.


\subsection{Small-gain analysis} \label{sect:small_gain}

\begin{definition} \label{def:io_gain}
	Fix $\beta \in \RR$.
	A $(U, X, Y)$-relation $S$ has \define{finite $\beta$-IO-gain} if there exists $\alpha \geq 0$ and a function $b : S \rightarrow \RRge{0}$ such that\footnote{The subscript $s$ in front of $[0,t]$ denotes the symbol used as the dummy variable.}
	\begin{align*}
		\|e^{\beta s} y(s)\|_{L^{\infty}({_s}[0,t], Y)} \leq \alpha \| e^{\beta s} u(s) \|_{L^{\infty}({_s}[0,t], U)} + b(u,x,y) \qquad \forall (u,x,y) \in S,\, t \geq 0.
	\end{align*}
	Equivalently,
	\begin{align*}
		\| e^{\beta t} y(t) \| \leq \alpha \| e^{\beta s} u(s) \|_{L^{\infty}({_s}[0,t], U)} + b(u,x,y) \qquad \forall (u,x,y) \in S,\, t \geq 0.
	\end{align*}
	The \define{$\beta$-IO-gain} of $S$ is the infimum over all $\alpha$ for which the inequality holds.
	Similarly, the \define{$\beta$-IS-gain} of $S$ is the infimum over all $\alpha$ satisfying
	\begin{align*}
		\| e^{\beta t} x(t) \| \leq \alpha \| e^{\beta s} u(s) \|_{L^{\infty}({_s}[0,t], U)} + b(u,x,y) \qquad
		\forall (u, x, y) \in S,\, t \geq 0.
	\end{align*}
\end{definition}

It is evident that the $\beta$-IO-gain $g(\beta)$ and the $\beta$-IS-gain $h(\beta)$ of $S$ are non-decreasing functions of $\beta$.
Furthermore, the $\beta$-IO-gain of $S_2 \circ S_1$ is bounded by $g_1(\beta) g_2(\beta)$, and the $\beta$-IS-gain of $S_2 \circ S_1$ is bounded by $h_1(\beta) + h_2(\beta) g_1(\beta)$.
Another trivial observation is the following version of the small-gain theorem.

\begin{lemma}[Small-gain] \label{lem:smallgain}
	Fix $\beta \in \RR$.
	If a $(U, X, U)$-relation $S$ has a finite $\beta$-IS-gain and a $\beta$-IO-gain strictly less than $1$, then $\sysclose(S)$ is $\beta$-exponentially stable.
\end{lemma}
\begin{proof}
	$(0, x, 0) \in \sysclose(S)$ iff $(u, x, u) \in S$ for some $u \in C^{0}_{\local}$.
	Since there exists $\alpha \in (0,1)$ such that
	\begin{align*}
		\|e^{\beta s} u(s)\|_{L^{\infty}({_s}[0,t], U)} \leq \alpha \| e^{\beta s} u(s) \|_{L^{\infty}({_s}[0,t], U)} + b(u,x,u) \qquad \forall t \geq 0,
	\end{align*}
	it follows that $e^{\beta t} u(t) \in L^{\infty}({_t}\RRge{0}, U)$.
	By finiteness of the $\beta$-IS-gain, $e^{\beta t} x(t) \in L^{\infty}({_t}\RRge{0}, X)$.
\end{proof}

Using the following definition of IS- and IO-gains, $\sysclose(S)$ is exponentially stable if $S$ has a finite IS-gain and an IO-gain less than $1$.

\begin{definition}
	Let $g(\beta)$ be the $\beta$-IO-gain of $S$, the IO-gain of $S$ is defined as $\inf_{\beta > 0} g(\beta) = \lim_{\beta \rightarrow 0^+} g(\beta)$.
	The IS-gain of $S$ is analogously defined.
\end{definition}

To apply \Cref{lem:smallgain} on the closed-loop system in \Cref{fig:closedloop_diagram}, we need estimates on the IO- and IS-gains of $R$ and $S_s$.
However, when $C$ is unbounded, $\system(A, B, C)$ may have infinite IO-gain even if it is exponentially stable.
This issue is resolved by considering IO-gains with respect to the $C^n$-norms:

\begin{definition}
	Fix $\beta \in \RR$ and $n, m \in \ZZge{0}$.
	A $(U, X, Y)$-relation $S$ has \define{finite $\beta$-IO-$(n,m)$-gain} if $\{y \mid (u, x, y) \in S,\, u \in C^{n}_{\local}\} \subseteq C^m_{\local}$ and there exists $\alpha \geq 0$ and $b : S \rightarrow \RRge{0}$ such that
	\begin{align*}
		\| e^{\beta t} \partial_t^k y(t) \| \leq \alpha \max_{j \in \{0, \ldots, n\}} \| e^{\beta s} \partial_s^j u(s) \|_{L^{\infty}({_s}[0,t], U)} + b(u,x,y)
	\end{align*}
	for all $k \in \{0, \ldots, m\}$, $t \geq 0$, and $(u, x, y) \in S$ satisfying $u \in C^{n}_{\local}$.
	The \define{$\beta$-IO-gain} of $S$ is the infimum over all $\alpha$ for which the inequality holds.
	Similarly, the \define{$\beta$-IS-gain} of $S$ is the infimum over all $\alpha$ satisfying
	\begin{align*}
		\| e^{\beta t} \partial_t^k x(t) \| \leq \alpha \max_{j \in \{0, \ldots, n\}} \| e^{\beta s} \partial_s^j u(s) \|_{L^{\infty}({_s}[0,t], U)} + b(u,x,y)
	\end{align*}
	for all $k \in \{0, \ldots, m\}$, $t \geq 0$, and $(u, x, y) \in S$ satisfying $u \in C^{n}_{\local}$.
	The IO-$(n,m)$-gain of $S$ is $\inf_{\beta > 0} g(\beta) = \lim_{\beta \rightarrow 0^+} g(\beta)$, where $g$ is the $\beta$-IO-$(n,m)$-gain of $S$.
	The IS-$(n,m)$-gain is analogously defined.
\end{definition}

The next result implies that every system of the form $\system(A, B, C)$ has a finite IO-$(1,0)$-gain.

\begin{lemma} \label{lem:io_gain_weak}
	If $\| e^{At} \| \leq a e^{-\alpha t}$ for some $a \geq 1$ and $\alpha \in \RR$, then the $\beta$-IS-$(1,1)$-gain $h(\beta)$ and the $\beta$-IO-$(1,0)$-gain $g(\beta)$ of $S = \system(A, B, C)$ satisfy
	\begin{align*}
		h(\beta) &\leq \max\left( \frac{a \|B\|}{\alpha - \beta},\, (a+1)\|B\| \right), \\
		g(\beta) &\leq a \|B\| \|C\|_{\domain(A) \rightarrow Y} \frac{1 + \alpha - \beta}{\alpha - \beta} 
	\end{align*}
	for all $\beta < \alpha$.
\end{lemma}
\begin{proof}
	Fix $(u, x, y) \in S$ satisfying $u \in C^1_{\local}$.
	By the assumptions $\| e^{At} \| \leq a e^{-\alpha t}$ and $\beta < \alpha$, the identity $x(t) = e^{At} x(0) + \int_{0}^{\infty} e^{A(t-s)} B u(s) ds$ yields the inequality 
	\begin{align*}
		\| e^{\beta t} x(t) \|_{X} &\leq a e^{-(\alpha - \beta)t} \|x(0)\|_{X}
		 + a \|B\| \|e^{\beta s} u(s)\|_{L^{\infty}([0,t])} \int_{0}^{t} e^{-(\alpha - \beta)(t-s)} ds \\
		&\leq a \|x(0)\|_{X} + \frac{a \|B\|}{\alpha - \beta} \|e^{\beta s} u(s)\|_{L^{\infty}([0,t])}.
	\end{align*}
	Since $x(0) \in \domain(A)$ and $u \in C^{1}_{\local}$, we get the expression $A x(t) = A e^{At} x(0) + \int_{0}^{t} e^{A(t-s)} B \dot{u}(s) ds + e^{At} B u(0) - B u(t)$ and the inequality
	\begin{align*}
		\| e^{\beta t} Ax(t) \|_{X} &\leq a \|Ax(0)\|_{X} + \frac{a \|B\|}{\alpha - \beta} \|e^{\beta s}\dot{u}(s)\|_{L^{\infty}([0,t])} + a\|B\| \|e^{\beta s}u(s)\|_{L^{\infty}([0,t])}.
	\end{align*}
	Combined with $\dot{x} = Ax + Bu$, we obtain the desired bound on the $\beta$-IS-$(1,1)$-gain $h(\beta)$.
	Since $\|x(t)\|_{\domain(A)} = \|Ax(t)\|_{X} + \|x(t)\|_{X}$, we deduce
	\begin{align*}
		\|e^{\beta t} x(t)\|_{\domain(A)} &\leq a \|x(0)\|_{\domain(A)} + \frac{1 + \alpha - \beta}{\alpha - \beta} a \|B\| \max_{j \in \{0,1\}} \|e^{\beta s} \partial_s^j u(s)\|_{L^{\infty}([0,t])}.
	\end{align*}
	Combined with $\|y(t)\| \leq \|C\|_{\domain(A) \rightarrow Y} \|x(t)\|_{\domain(A)}$, we obtain desired bound on the IO-$(1,0)$-gain $g(\beta)$.
\end{proof}

In the case $A$ is bounded, $\system(A, B, C)$ has finite IO-$(0,1)$-gains.

\begin{lemma} \label{lem:io_gain_strong}
	If $\|e^{At}\| \leq a e^{-\alpha t}$ for some $a \geq 1$ and $\alpha \in \RR$,
	then the $\beta$-IS-$(0,1)$-gain $h(\beta)$ and the $\beta$-IO-$(0,1)$-gain $g(\beta)$ of $S = \system(A, B, C)$ satisfy 
	\begin{align*}
		h(\beta) &\leq \max\left(\frac{a\|B\|}{\alpha - \beta},\, \frac{a\|B\| \|A\|}{\alpha - \beta} + \|B\| \right), \\
		g(\beta) &\leq \max\left(\frac{a\|B\|\|C\|}{\alpha - \beta},\, \frac{a\|B\| \|A\| \|C\|}{\alpha - \beta} + \|B\| \|C\| \right)
	\end{align*}
	for all $\beta < \alpha$.
\end{lemma}
\begin{proof}
	The $\beta$-IS-$(0,0)$-gain of $S$ is bounded by $a\|B\|/(\alpha - \beta)$. 
	The bound on $h(\beta)$ follows from the equation $\dot{x} = Ax + Bu$.
\end{proof}

In view of the results above, exponential stability of the closed-loop system $\sysclose(S \circ T)$ in \Cref{fig:closedloop_diagram} can be established using the following Proposition.

\begin{proposition} \label{pro:smallgain} \stepcounter{refer}
	Given a $(U, X, Y)$-relation $S = \system(A, B, C)$ and a finite-dimensional $(Y, W, U)$-relation $T = \system(E, F, G)$.
	Assume that 
	\begin{itemize} \itemsep 0mm
		\item $S$ has finite IS-$(1,1)$-gain and $T$ has finite IS-$(0,1)$-gains;
		\item The product of the IO-$(1,0)$-gain of $S$ and the IO-$(0,1)$-gain of $T$ is strictly less than $1$.
	\end{itemize}
	Then $\sysclose(T \circ S)$ is exponentially stable.
\end{proposition}
\begin{proof}
	By assumption, there exists $\beta > 0$ such that the $\beta$-IO-$(1,1)$-gain $g(\beta)$ of $T \circ S$ is less than $1$.
	Therefore, for every behaviour $(u, (x,w), u) \in T \circ S$, there is a number $b > 0$ such that
	\begin{align*}
		\sup_{s \in [0, t]} \max(e^{\beta t} u(s),\, e^{\beta t} \dot{u}(s)) &\leq g(\beta) \sup_{s \in [0, t]} \max(e^{\beta t} u(s),\, e^{\beta t} \dot{u}(s)) + b \qquad \forall t \geq 0.
	\end{align*}
	This implies $e^{\beta t} u \in L^{\infty}$ and $e^{\beta t} \dot{u} \in L^{\infty}$.
	In view of hypotheses on the IS-gains, $e^{\beta t} x \in L^{\infty}$, $e^{\beta t} \dot{x} \in L^{\infty}$, $e^{\beta t} w \in L^{\infty}$, $e^{\beta t} \dot{w} \in L^{\infty}$.
	
	By \Cref{pro:closedloop_semigroup}, $\sysclose(T \circ S) \cong \system(H,0,0)$ for some generator $H$ on $X \times W$.
	The previous analysis implies $\|e^{Ht} x\|_{\domain(H)} \lesssim e^{\beta t}$ for every $x \in \domain(H)$.
	By the lemma below, we conclude that $\system(H,0,0)$ is exponentially stable.
\end{proof}

\begin{lemma}
	Let $A$ be a semigroup generator on a Banach space $X$.
	Assume that for every $x \in \domain(A)$, $\|e^{At}x\|_{\domain(A)} \lesssim e^{\alpha t}$ uniform in $t \geq 0$, then $\|e^{At}\| \lesssim e^{\alpha t}$ uniform in $t \geq 0$.
\end{lemma}
\begin{proof}
	Let $e^{A_1 t} : \domain(A) \rightarrow \domain(A)$ represents the restriction of the semigroup $e^{At}$ to $\domain(A)$. 
	By assumption, the family of bounded operators $\{e^{-\alpha t} e^{A_1 t} \mid t \geq 0\}$ satisfies the pointwise bound $\sup_{t \geq 0} \|e^{-\alpha t} e^{A_1 t} x\| < \infty$.
	By the uniform boundedness principle, $\sup_{t \geq 0} \|e^{-\alpha t} e^{A_1 t}\|_{\domain(A) \rightarrow \domain(A)} < \infty$.
	Fix any $\lambda \in \CC \setminus \spec(A)$, the operator $A - \lambda$ is an isomorphism from $\domain(A)$ to $X$. 
	By the commutation diagram
	\begin{equation*} \begin{tikzcd}
		\domain(A) \arrow[r, "e^{A_1 t}"] \arrow[d, hook, two heads, "A - \lambda"] & \domain(A) \arrow[d, hook, two heads, "A - \lambda"] \\
		X \arrow[r, "e^{At}"] & X
	\end{tikzcd} \end{equation*}
	we conclude that $\|e^{At}\|_{X \rightarrow X} \lesssim \|e^{A_1 t}\|_{\domain(A) \rightarrow \domain(A)} \lesssim e^{\alpha t}$.
\end{proof}

\subsection{A general condition for feedback stabilization} \label{sect:control_design}

To ensure exponential stability of the closed-loop system in \Cref{fig:closedloop_diagram}, we will apply \Cref{pro:smallgain} to $R$ and $S_s$. 
By construction, both systems are exponentially stable, hence have finite IS-gains.
It remains to ensure that the IO-gain of $R \circ S_s$ is sufficiently small.
Since $R$ is composed of $S_u$, which always contain the unstable modes of $S$, we cannot expect to reduce the IO-gain of $R$ arbitrarily.
The only obvious alternative is to reduce the IO-gain of $S_s$ by choosing the decomposition $S \cong S_u + S_s$.
This motivates the following definition.

\begin{definition} \label{def:quasi_finite}
	A system $S = \system(A, B, C)$ is \define{quasi-finite} if for every $\epsilon > 0$, there is a decomposition $S \cong S_u + S_s$ for a finite-dimensional $S_u = \system(A_u, B_u, C_u)$ and an exponentially stable $S_s = \system(A_s, B_s, C_s)$ such that the IO-$(1,0)$-gain of $\system(A_s, B_s, C_s)$ is less than $\epsilon$.
\end{definition}

The quasi-finite assumption turns out to be sufficient for solving \Cref{prob:control}.

\begin{theorem} \label{thm:control} \stepcounter{refer}
	Given any stabilizable, detectable, quasi-finite $(U, X, Y)$-relation $S = \system(A, B, C)$, there exists a finite-dimensional system $T = \system(E, F, G)$ such that $\sysclose(T \circ S)$ is exponentially stable.
\end{theorem}
\begin{proof}
	For every subscript $k$, the operators $(A_k, B_k, C_k)$ will denote the unique operators such that $S_k = \system(A_k, B_k, C_k)$.
	Let $X_{\bar{u}} \oplus X_{\bar{s}}$ be the canonical stable/unstable decomposition of the state space $X$ defined in \Cref{pro:finite_unstable2}.
	In this coordinate system, $S$ writes $S_{\bar{u}} + S_{\bar{s}}$.
	By \Cref{pro:finite_unstable2} (e), for every decomposition $S \cong S_u + S_s$ with an exponentially stable system $S_s$, the subsystem $S_u$ further decomposes into $S_{\bar{u}} + S_r$.
	Using the decomposition $S = S_{\bar{u}} + S_r + S_s$, the closed-loop $\sysclose(T \circ S)$ expands into \Creflocal{fig:1}.
	\begin{figure}[h]
	\centering \begin{tikzpicture}
		\tikzstyle{sys}=[rectangle, draw, minimum width=10mm, minimum height=8mm];
		\node[sys] at (0,0) (S_r) {$S_r$};
		\node[sys] at (0,1.2) (S_u) {$S_{\bar{u}}$};
		\node[sys] at (0,-1.2) (S_s) {$S_s$};
		\node[sys] at (0,2.4) (T) {$T$};
		\node[circle, draw, scale=0.7] at (1.5, 0) (out) {$+$};
		\coordinate (in) at (-1.5, 0);
		
		\draw[->] (S_u) -|node[above, pos=0.25] {$y_{\bar{u}}$} (out);
		\draw[->] (S_r) --node[above] {$y_r$} (out);
		\draw[->] (S_s) -|node[above, pos=0.25] {$y_s$}  (out);
		\draw[->] (out) -| ($(out) + (0.5, 0)$) |- (T);
		
		\draw[->] (T) -| (in) -- (S_r);
		\draw[->] (in) |- (S_u);
		\draw[->] (in) |- node[above, pos=0.75] {$u$} (S_s);
		
		\draw[dotted, thick] (-2, -0.6) -- (2.5, -0.6) --node[right] {$R$} (2.5, 3) -- (-2, 3) -- cycle;
	\end{tikzpicture}
	\caption{Closed-loop system $\sysclose(T \circ S)$ in \Cref{thm:control}} \labellocal{fig:1}
	\end{figure}
	
	Quasi-finiteness of $S$ means that the IO-$(1,0)$-gain of $\system(A_s, B_s, C_s)$ can be chosen arbitrarily small, though the decomposition will affect $S_r$ as well.
	Assuming that we can show
	\begin{enumerate}[(i)] \itemsep 0mm
		\item\labellocal{item:i} There exists $b > 0$ such that for every $S_r$, we can design $T$ such that the system $R$ in \Creflocal{fig:1}, with input $y_s$ and output $u$, is exponentially stable and has an IO-$(0,1)$-gain less than $b$,
	\end{enumerate}
	then $\sysclose(R \circ S_s)$ is exponentially stable by \Cref{pro:smallgain}.
	
	``Proof of \reflocal{item:i}''	
	Our objective is to stabilize $S_u \cong S_{\bar{u}} + S_r$, while still fulfilling an opportune small-gain property on the interconnection $\sysclose(R \circ S_s)$. 
	In particular, the design of the controller $T$ should not impact the intrinsic stability of $S_r$. 
	To this end, write $S_{\bar{u}}+S_r$ as the ODE
	\begin{align*}
		\partial_t \begin{bmatrix} x_{\bar{u}} \\ x_r \end{bmatrix} &= 
		\begin{bmatrix}
			A_{\bar{u}} & 0\\
			0 & A_r
		\end{bmatrix} 
		\begin{bmatrix} x_{\bar{u}} \\ x_r \end{bmatrix} + \begin{bmatrix} B_{\bar{u}} \\ B_r \end{bmatrix} u, \\
		y_u &= \begin{bmatrix} C_{\bar{u}} & C_r \end{bmatrix} \begin{bmatrix} x_{\bar{u}}\\ x_r \end{bmatrix}
	\end{align*}
	and consider the observer-based controller $T$ defined by
	\begin{align*}
		\partial_t \begin{bmatrix} w_{\bar{u}} \\ w_r \end{bmatrix} &= 
		\left( \begin{bmatrix}
		 A_{\bar{u}} & 0\\
		 0 & A_r
		\end{bmatrix} + \begin{bmatrix} L_{\bar{u}} \\ 0 \end{bmatrix} \begin{bmatrix} C_{\bar{u}} & C_r \end{bmatrix} \right)
		\begin{bmatrix} w_{\bar{u}} \\ w_r \end{bmatrix} - \begin{bmatrix} L_{\bar{u}} \\ 0 \end{bmatrix} (y_u + y_s) + \begin{bmatrix} B_{\bar{u}} \\ B_r \end{bmatrix} u, \\
		u &= \begin{bmatrix} K_{\bar{u}} & 0 \end{bmatrix} \begin{bmatrix} w_{\bar{u}} \\ w_r \end{bmatrix}.
	\end{align*}
	In the state vector $[x_{\bar{u}}^\top, x_r^\top, e_{\bar{u}}^\top, e_r^\top]^\top$, where $e_{\bar{u}}=w_{\bar{u}}-x_{\bar{u}}$ and $e_{r}=w_{r}-x_{r}$, the dynamics of system $R$ reads as
	\begin{align*}
		\partial_t \begin{bmatrix} x_{\bar{u}} \\ x_r \\ e_{\bar{u}} \\ e_r \end{bmatrix} &= 
		\setlength\arraycolsep{2pt}
		\begin{bmatrix} A_{\bar{u}} + B_{\bar{u}} K_{\bar{u}} & 0 & B_{\bar{u}} K_{\bar{u}} & 0 \\ B_rK_{\bar{u}} & A_r & B_rK_{\bar{u}} & 0 \\ 0 & 0 & A_{\bar{u}} + L_{\bar{u}} C_{\bar{u}} & L_{\bar{u}}C_r \\ 0 & 0 & 0 & A_r \end{bmatrix}
		\begin{bmatrix} x_{\bar{u}} \\ x_r \\ e_{\bar{u}} \\ e_r \end{bmatrix} - 
		\begin{bmatrix} 0 \\ 0 \\ L_{\bar{u}} \\ 0 \end{bmatrix} y_s, \\
		u &= \begin{bmatrix} K_{\bar{u}} & 0 & K_{\bar{u}} & 0 \end{bmatrix} 
		\begin{bmatrix} x_{\bar{u}} \\ x_r \\ e_{\bar{u}} \\ e_r \end{bmatrix}.
	\end{align*}
	Since $x_r$ does not affect the output and $e_r$ is autonomous with $A_r$ Hurwitz, the IO-$(1,0)$-gain of $R$, with input $y_s$ and output $u$, is equal to the IO-$(1,0)$-gain of
	\begin{align*}
		\system\left( \begin{bmatrix} A_{\bar{u}} + B_{\bar{u}} K_{\bar{u}} & B_{\bar{u}} K_{\bar{u}} \\ 0 & A_{\bar{u}} + L_{\bar{u}} C_{\bar{u}} \end{bmatrix}, \begin{bmatrix} 0 \\ -L_{\bar{u}} \end{bmatrix}, \begin{bmatrix} K_{\bar{u}} & K_{\bar{u}} \end{bmatrix} \right).
	\end{align*}
	By selecting $L_{\bar{u}}$ and $K_{\bar{u}}$ such that $A_{\bar{u}} + L_{\bar{u}} C_{\bar{u}}$ and $A_{\bar{u}} + B_{\bar{u}} K_{\bar{u}}$ are Hurwitz, the system $R$ is exponentially stable and has a finite IO-$(1,0)$-gain is independent of $S_r$.
\end{proof}

\subsection{Quasi-finiteness} \label{sect:quasifinite}

In this section, we present some sufficient conditions for quasi-finiteness, the property appearing in the hypothesis of \Cref{thm:control}. 
\Cref{thm:control_schauderbasis} will be proved as a corollary of \Cref{thm:control}.

Assume $S \cong S_u + S_s$ is a decomposition into finite-dimensional and exponentially stable systems. 
By \Cref{lem:io_gain_weak}, the IO-$(1,0)$-gain of $S_s = \system(A_s, B_s, C_s)$ is proportional to $\|B_s\|$.
Quasi-finiteness can therefore be recovered by an assumption on the norm $\|B_s\|$.

\begin{proposition} \label{pro:quasifinite_1}
	Consider a $(U, X, Y)$-relation $S = \system(A, B, C)$ with finite unstable part.
	Assuming that for every $\epsilon > 0$, there exists a decomposition of $X$ into closed $A$-invariant subspaces $X_u \oplus X_s$ such that $\dim X_u < \infty$ and $\| P_s B \| < \epsilon$ for the projection $P_s$ onto $X_s$, then $S$ is quasi-finite.
\end{proposition}
\begin{proof}
	We may assume wlog that $S$ is exponentially stable.
	If not, write $S \cong S_{\bar{u}} + S_{\bar{s}}$ using \Cref{pro:finite_unstable2} and apply this result to $S_{\bar{s}}$.

	In coordinates $X_u \oplus X_s$, write
	\begin{align*}
		A &= \begin{bmatrix} A_u & 0 \\ 0 & A_s \end{bmatrix},
		\quad
		B = \begin{bmatrix} B_u \\ B_s \end{bmatrix}, 
		\quad
		C = \begin{bmatrix} C_u & C_s \end{bmatrix}.
	\end{align*}
	Then $S \cong S_u + S_s$ for the $(U, X_u, Y)$-relation $S_u = \system(A_u, B_u, C_u)$ and the $(U, X_s, Y)$-relation $S_s = \system(A_s, B_s, C_s)$.
	The IO-$(1,0)$-gain of $S_s$ is bounded by the product of the IO-$(1,0)$-gain of the $(X_s, X_s, Y)$-relation $\system(A_s, I, C_s)$ and the IO-$(1,1)$-gain of the $(U, 0, X_s)$-relation $B_s$.
	The IO-$(1,0)$-gain of $\system(A_s, I, C_s)$ is bounded by the IO-$(1,0)$-gain of $\system(A, I, C)$, a finite number independent of the decomposition $X = X_u \oplus X_s$.
	The IO-$(1,1)$-gain of $B_s$ is equal to $\|P_s B\|$.
	Therefore, the IO-$(1,0)$-gain of $S_s$ can be arbitrarily small.
\end{proof}

The hypothesis of \Cref{pro:quasifinite_1} is satisfied when $X$ has a Schauder subspace-basis adapted to $A$ (recall \Cref{def:schauderbasis}).

\begin{proposition} \label{pro:quasifinite_2}
	Let $S = \system(A, B, C)$ be a $(U, X, Y)$-relation with finite unstable part and a finite-rank operator $B$.
	If $X$ has a Schauder subspace-basis adapted to $A$, then $S$ is quasi-finite.
\end{proposition}
\begin{proof}
	Let $\{X_j\}_{j \in \NN}$ be a Schauder subspace-basis adapted to $A$.
	By \Cref{lem:schauder_basis} below, for every finite or cofinite $N \subseteq \NN$, $\bigoplus_{j \in N} X_j$ is a closed $A$-invariant subspace of $X$.
	$\Ima B$ is spanned by a finite number of vectors $b_1, \ldots, b_n$. 
	Apply \Cref{pro:quasifinite_1} with $X_u = \bigoplus_{j \in N} X_j$ and $X_s = \bigoplus_{j \not\in N} X_j$, where $N$ is large enough such that every $b_j$ is well-approximated by its projection on $X_u$.
\end{proof}

\begin{lemma} \label{lem:schauder_basis}
	Let $\{X_j\}_{j \in \NN}$ be a Schauder subspace-basis of a Banach space $X$.
	The norm $\|x\|_{X}$ is equivalent to 
	\begin{talign*}
		\|x\|_{*} := \left\{\sup_{n\in\NN} \left\|\sum_{j\leq n} x_j\right\|_{X} \mid \;
		\begin{aligned}
			&x = \sum_{j\in\NN} x_j \\
			&x_j \in X_j
		\end{aligned}	
		\right\}.
	\end{talign*}
	In particular, the projection $X \rightarrow \bigoplus_{j \in N} X_j \subseteq X$ is continuous when $N$ a is finite or cofinite subset of $\NN$.
\end{lemma}
\begin{proof}
	Since finite-dimensional topological vector spaces are complete, each $X_j$ is closed in $X$.
	One can verify that $X$ is a Banach space with respect to $\| \cdot \|_{*}$.
	The canonical map $(X, \|\cdot\|_{*}) \rightarrow (X, \|\cdot\|_{X})$ is a continuous bijection between Banach spaces, hence a homeomorphism.
\end{proof}

In view of \Cref{pro:quasifinite_2} and \Cref{thm:control}, we obtain \Cref{thm:control_schauderbasis} stated at the beginning of this section by following the next remaining steps.
Write $S \cong S_u + S_s$ where $S_u$ is a finite-dimensional and $S_s$ is exponentially stable, by \Cref{lem:stabilizability}, $S_u = \system(A_u, B_u, C_u)$ is stabilizable and detectable. 
	Since there exists finite-rank $F$ such that $S_u \circ F = \system(A_u, B_u F, C_u)$ remains stabilizable, $S \circ F$ is stabilizable, detectable, and has finite unstable part. By \Cref{pro:quasifinite_2}, $S \circ F$ is quasi-finite, so \Cref{prob:control} is solvable for $S \circ F$, which completes the proof.

\section{Examples} \label{sect:applications}

\subsection{Parabolic equation}

A prototypical parabolic equation is the heat equation
\begin{align} \label{eq:heat}
	\left\{
	\begin{aligned}
		\partial_t x(t, \xi) &= \Delta_{\xi} x(t, \xi) \quad &&\text{on } \RRge{0} \times \Omega, \\
		x(t, \xi) &= 0 \quad &&\text{on } \RRge{0} \times \Omega_1, \\
		\hat{n} \cdot \gradient_{\xi} x(t, \xi) &= 0 \quad &&\text{on } \RRge{0} \times \Omega_2,
	\end{aligned}
	\right.
\end{align}
where $\Omega$ is a compact $d$-dimensional submanifold of $\RR^d$ with smooth boundary $\partial \Omega$, which is the union of disjoint closed sets $\Omega_1$, $\Omega_2$, and $\hat{n}$ is the outward pointing normal vector on the boundary.
A weak formulation of \Cref{eq:heat} is the abstract Cauchy problem $\dot{x} = Ax$, where $A$ is an unbounded operator on $L^2(\Omega)$ defined by
\begin{align} \label{eq:laplace_op}
	\begin{aligned}
	\domain(A) &= \left\{x \in H^2(\Omega) \mid \, x = 0 \text{ in } H^{3/2}(\Omega_1), \; \hat{n} \cdot \gradient x = 0 \text{ in } H^{1/2}(\Omega_2) 
	\right\}, \\
	A &= \Delta x.
	\end{aligned}
\end{align}
The operator $A$ is sectorial \cite[Chapter 7]{pazy83}, has compact resolvents (because $H^2(\Omega) \hookrightarrow L^2$ is compact by Rellich-Kondrachov), and is self-adjoint (as a symmetric semigroup generator).
More generally, typical parabolic systems have the form $\system(A, B, C)$ for some sectorial operator $A$ with compact resolvents.
Such systems always have finite unstable part.

\begin{lemma} \label{lem:sectorial_fup}
	Any sectorial operator $A$ with compact resolvents has finite unstable part.
\end{lemma}
\begin{proof}
	Since $A$ has compact resolvents, its spectrum consists of isolated eigenvalues.
	Since $A$ is sectorial, $\spec(A) \cap \CCge{0}$ is finite.
	Let $\Gamma$ be any positively oriented simple closed curve encircling only the eigenvalues of $A$ in $\CCge{0}$.
	Define the projection operators (see \cite[Lemma 2.5.7]{curtain95} for their properties)
	\begin{align*}
		P_u = \frac{1}{2\pi i} \int_{\Gamma} (z - A)^{-1} dz, 
		\qquad
		P_s = 1 - P_u,
	\end{align*}
	we claim that the desired decomposition in \Cref{def:finite_unstable} is given by $\domain(e^A) = X_s \oplus X_u = P_s \domain(e^A) \oplus P_u \domain(e^A)$.
	
	Let $A_s$ and $A_u$ be the restrictions of $A$ to $X_s$ and $X_u$.
	We know from the properties of the projections that $\spec(A_u) = \spec(A) \cap \CCge{0}$, and $X_u$ is the sum of the unstable generalized eigenspaces of $A$.
	It remains to prove that $A_s$ is exponentially stable. 
	Since $A_s$ is sectorial, $e^{A_s}$ can be expressed using sectorial functional calculus as $\frac{1}{2\pi i} \int_{\Gamma} e^{z} (z - A_s)^{-1} dz$, for some suitable $\Gamma$.
	By the resolvent identity, $e^{A_s}$ is equal to
	\begin{align*}
		\frac{1}{2\pi i} (\omega - A_s)^{-1} \int_{\Gamma} e^{z} (1 + (\omega - z) (z - A_s)^{-1}) dz
	\end{align*}
	for some fixed $\omega \not\in \spec(A_s)$.
	This implies that $e^{A_s}$ is compact, hence has pure point spectrum.
	By the spectral mapping property of point spectrum \cite[Chpt 2, Thm 2.4]{pazy83}, $\spec(e^{A_s}) \cup \{0\} = e^{\spec(A_s)}\cup \{0\} \subseteq \{ z \in \CC \mid |z| < 1 \}$.
	We conclude that the spectral radius of $e^A$ is strictly less than $1$, and that $\| e^{At} \|$ decays exponentially by the spectral radius formula.
\end{proof}

If we can show that $\domain(e^A)$ has a Schauder subspace-basis adapted to $A$, then the parabolic system $\system(A, B, C)$ is exponentially stabilized (by a finite-dimensional linear controller) if and only if it is stabilizable and detectable.
Since $A$ has compact resolvents, the Schauder subspace-basis assumption is verified in particular when $A$ is self-adjoint.
In this case, the eigenspaces of $A$ form a Schauder subspace-basis adapted to $A$. 
When $A$ is not self-adjoint, the assumption may be verified by perturbation methods such as \cite[XIX.2.7]{dunford71} or \cite[Chpt V, Thm 4.15a]{kato80}.

\begin{example}[Self-adjoint heat equation] \label{ex:parabolic_1} \stepcounter{refer}
	Consider the parabolic equation
	\begin{align*}
		\left\{
		\begin{aligned}
		\partial_t x(t, \xi) &= \partial_{\xi}^2 x + b x + f(\xi) u(t) \quad \text{on } \RRge{0} \times [0,1], \\
		\partial_{\xi} x(t, 0) &= \partial_{\xi} x(t, 1) = 0, \\
		y(t) &= x(t, 0),
		\end{aligned}
		\right.
	\end{align*}
	where $f \in L^2([0,1])$.
	The weak formulation of this equation is the abstract Cauchy problem on state space $X = L^2([0,1])$:
	\begin{align} \labellocal{eq:acp}
		\begin{aligned}
		\dot{x}(t) &= A x(t) + B u(t), \\
		y(t) &= C x(t)
		\end{aligned}
	\end{align}
	where $\domain(A) = \{x \in H^2([0,1]) \mid \partial_{\xi} x(0) = \partial_{\xi}x(1) = 0\}$, $Ax = \partial_{\xi}^2 x$, $\domain(C) = \domain(A)$, $Bu = fu$, and $Cx = x(0)$.
 	Since $A$ is a self-adjoint sectorial operator with compact resolvents, \Creflocal{eq:acp} can be exponentially stabilized (in $L^2$) iff it is stabilizable and detectable.
	Since \Creflocal{eq:acp} has finite unstable part, it is stabilizable and detectable iff its projection to every generalized eigenspace is stabilizable and detectable.
	By direct computations, $\spec(A) = \{b-\pi^2 k^2\}_{k \in \ZZge{0}}$ and $\Ker (b - \pi^2 k^2 - A) = \Span\{ \cos(\pi k \xi) \}$.
	The projection of \Creflocal{eq:acp} to $\Span\{ \cos(\pi k \xi) \}$ is the one-dimensional system
	\begin{align*}
		\dot{x} &= (b - \pi^2 k^2) x + \frac{\langle \cos(\pi k \xi), f \rangle_{L^2}}{\langle \cos(\pi k \xi), \cos(\pi k \xi) \rangle_{L^2}} \cdot u, \\
		y &= x.
	\end{align*}
	This projected system is always detectable and stabilizable iff either $b - \pi^2 k^2 < 0$ or $\langle \cos(\pi k \xi), f \rangle_{L^2} \neq 0$.
	Therefore, \Creflocal{eq:acp} can be exponentially stabilized by a finite-dimensional controller iff $\langle \cos(\pi k \xi), f \rangle_{L^2} \neq 0$ for every $k \in \ZZge{0}$ such that $\pi^2 k^2 \leq b$.
\end{example}

\begin{example}[Heat equation with boundary input] \stepcounter{refer}
	Consider the equation
	\begin{align} \labellocal{eq:parabolic}
		\left\{
		\begin{aligned}
		\partial_t x(t, \xi) &= \partial_{\xi}^2 x + b x + f(\xi) u(t) \quad \text{on } \RRge{0} \times [0,1], \\
		\partial_{\xi} x(t, 0) &= 0, \\
		\partial_{\xi} x(t, 1) &= u(t), \\
		y(t) &= x(t, 0).
		\end{aligned}
		\right.
	\end{align}
	Due to the presence the input at the boundary, this equation does not take the form $\system(*, *, *)$.
	This issue can be resolved by a change of variable.
	Let $h \in H^2([0,1])$ and $a \in \CC$ be any pair satisfying $(\partial_{\xi}^2 + b) h = a h,\, \partial_{\xi}h(0) = 0,\, \partial_{\xi}h(1) = 1$, then $(u, x, y)$ solves \Creflocal{eq:parabolic} if and only if $z(t, \xi) = x(t, \xi) - h(\xi) u(t)$ solves
	\begin{align*}
		\left\{
		\begin{aligned}
		\partial_t z(t, \xi) &= \partial_{\xi}^2 z + b z + f(\xi) u + a h(\xi) u - h(\xi) \partial_t u, \\
		\partial_{\xi} z(t, 0) &= 0, \\
		\partial_{\xi} z(t, 1) &= 0, \\
		y(t) &= z(t, 0) + h(0) u(t).
		\end{aligned}
		\right.
	\end{align*}
	This parabolic equation can be interpreted as a system with input $v = \partial_t u$, state $(u, z)$, and output $y$,
	which is the abstract Cauchy problem on the state space $\CC \times L^2([0,1])$:
	\begin{align} \labellocal{eq:acp}
		\begin{aligned}
		\partial_t \begin{bmatrix} u \\ z \end{bmatrix} &= \begin{bmatrix} 0 & 0 \\ f + ah & A \end{bmatrix} \begin{bmatrix} u \\ z \end{bmatrix} 
		+ \begin{bmatrix} 1 \\ -h \end{bmatrix} v, \\
		y &= \begin{bmatrix} h(0) & C \end{bmatrix} \begin{bmatrix} u \\ z \end{bmatrix}
		\end{aligned}
	\end{align}
	where $A, C$ are defined in \Cref{ex:parabolic_1}.
	To stabilize \Creflocal{eq:parabolic}, it suffices to stabilize \Creflocal{eq:acp}.
	
	We will characterize exponential stabilizability of \Creflocal{eq:acp}.
	The spectrum of the generator $M$ of \Creflocal{eq:acp} consists of $\spec(A)$ and an additional eigenvalue at $0$.
	At any nonzero $b - \pi^2 k^2 \in \spec(A)$, the corresponding eigenspace of $M$ is $\Span\{\begin{bmatrix} 0 & \cos(\pi k \xi) \end{bmatrix}^\top \}$.
	The $0$-eigenspace of $M$ is 
	\begin{align*}
		\Span\left\{\begin{bmatrix} 1 & g_1 \end{bmatrix}^\top \right\} + \Span \left\{ \begin{bmatrix} 0 & \cos(\pi k \xi) \end{bmatrix}^\top \mid k \in \ZZge{0} ,\, b - \pi^2 k^2 = 0 \right\}
	\end{align*}
	where $g_1$ is any vector in $\domain(A)$ such that $f + ah + Ag_1 =: g_2 \in \Ker A$.
	The vector $g_1$ exists because $A$ is invertible on $(\Ker A)^{\perp}$, the image of the spectral projection to $\spec(A) \setminus \{0\}$. 
	From this, we see that $M$ has finite unstable part and has an adapted Schauder subspace-basis. 
	Therefore, \Creflocal{eq:acp} can be exponentially stabilized iff it is stabilizable and detectable.
	This is verified by projecting \Creflocal{eq:acp} onto its eigenspaces.
	
	The projection onto the $0$-eigenspace is the two-dimensional system
	\begin{align*}
		\begin{aligned}
		\partial_t \begin{bmatrix} x_1 \\ x_2 \end{bmatrix} &= \begin{bmatrix} 0 & 0 \\ \frac{\langle \cos(\pi k \xi), g_2 \rangle_{L^2}}{\langle \cos(\pi k \xi), \cos(\pi k \xi) \rangle_{L^2}} & 0 \end{bmatrix} \begin{bmatrix} x_1 \\ x_2 \end{bmatrix} + \begin{bmatrix} 1 \\ -\frac{\langle \cos(\pi k \xi), h + g_1 \rangle_{L^2}}{\langle \cos(\pi k \xi), \cos(\pi k \xi) \rangle_{L^2}} \end{bmatrix} v, \\
		y &= \begin{bmatrix} h(0) + C g_1 & 1 \end{bmatrix} \begin{bmatrix} x_1 \\ x_2 \end{bmatrix}
		\end{aligned}
	\end{align*}
	if there exists $k \in \ZZge{0}$ such that $\pi^2 k^2 = b$; or the one-dimensional system $\dot{x} = v$, $y = (h(0) + C g_1) x$ if no such $k$ exists.
	The projection onto the eigenspace of a nonzero eigenvalue $b - \pi^2 k^2$ is the system
	\begin{align*}
		\dot{x} &= (b - \pi^2 k^2) x - \frac{\langle \cos(\pi k \xi), h + g_1 \rangle_{L^2}}{\langle \cos(\pi k \xi), \cos(\pi k \xi) \rangle_{L^2}} \cdot v, \\
		y &= x.
	\end{align*}
	These projected systems are stabilizable and detectable iff
	\begin{align} \labellocal{eq:constraint}
		\left\{
		\begin{aligned}
			\langle \cos(\pi k \xi), h + g_1 \rangle_{L^2} &\neq 0 \qquad \text{for } k \in \ZZge{0},\, \pi^2 k^2 < b, \\
			\langle \cos(\pi k \xi), g_2 \rangle_{L^2} &\neq 0 \qquad \text{for } k \in \ZZge{0},\, \pi^2 k^2 = b, \\
			h(0) + C g_1 &\neq 0.
		\end{aligned}
		\right.
	\end{align}
	One can show that this finite number of inequalities is simultaneously verified by a generic choice of $a > b$ and corresponding $h(\xi) = \cosh(\sqrt{a - b} \cdot \xi) / \sinh(\sqrt{a-b})$.
	Therefore, \Creflocal{eq:parabolic} can be exponentially stabilized regardless of the parameters $b$ and $f$.
	
	The existence of some parameter $a$ fulfilling \Creflocal{eq:constraint} can be justified as follows.
	Note that the mapping $(b, \infty) \rightarrow L^2([0,1]) : a \mapsto h$ is holomorphic.
	Furthermore, set $g_2$ be the projection of $f+ah$ to $\Ker A$, and $g_1 = \restr{A}{(\Ker A)^{\perp}}^{-1} (g_2 - f-ah)$, then the functions $a \mapsto g_1$ and $a \mapsto g_2$ are also holomorphic.
	Therefore, every expression in \Creflocal{eq:constraint} is holomorphic in $a$.
	Since the zeros of non-trivial holomorphic functions are isolated, it suffices to show that these expressions do not vanish identically.
	This can be checked explicitly.
	For example, for $k \in \ZZge{0}$ satisfying $\pi^2 k^2 < b$, the expression
	\begin{align*}
		\langle \cos(\pi k \xi), h + g_1 \rangle_{L^2} &= \langle \cos(\pi k \xi), h \rangle_{L^2} + \frac{1}{b - \pi^2 k^2} \langle \cos(\pi k \xi) , g_2 - f - ah \rangle_{L^2} \\
		&= \left( \frac{b - \pi^2 k^2 - a}{b - \pi^2 k^2} \right) \langle \cos(\pi k \xi), h \rangle_{L^2} - \frac{1}{b - \pi^2 k^2} \langle \cos(\pi k \xi), f \rangle_{L^2}
	\end{align*}
	is non-zero at some $a > b$.
\end{example}

\subsection{Hyperbolic equation}

In this section, we examine wave equations of the form
\begin{align} \label{eq:wave}
	\left\{
	\begin{aligned}
		(\partial_t^2 + \kappa \partial_t) x(t, \xi)  &= \Delta_{\xi}^2 x(t, \xi) \quad &&\text{on } \RRge{0} \times \Omega \\
		x(t, \xi) &= 0 \; &&\text{on } \RRge{0} \times \Omega_1 \\
		\hat{n} \cdot \gradient_{\xi} x(t, \xi) &= 0 \; &&\text{on } \RRge{0} \times \Omega_2
	\end{aligned}
	\right.
\end{align}
where $\kappa \in \RR$, and $\Omega$ is a compact $d$-dimensional submanifold of $\RR^d$ with smooth boundary $\partial \Omega$, which is the union of disjoint closed sets $\Omega_1$, $\Omega_2$.
A weak formulation of this equation is the second-order abstract Cauchy problem $\partial_t^2 x + \partial_t x = Ax$
where $A$ is the unbounded operator defined in \Cref{eq:laplace_op}.
Letting $v = \partial_t x$, we have the equivalent abstract Cauchy problem
\begin{align*}
	\partial_t \begin{bmatrix} x \\ v \end{bmatrix} &= \underbrace{\begin{bmatrix} 0 & 1 \\ A & -\kappa \end{bmatrix}}_{M} \begin{bmatrix} x \\ v \end{bmatrix}.
\end{align*}
It is known that $M$ generates a strongly continuous group.
This result can be stated generally as

\begin{proposition} \stepcounter{refer}
	Let $A$ be a self-adjoint generator of a semigroup on $X$.
	Fix $\omega > 0$ such that $\omega - A$ is strictly positive, hence $(\omega - A)^{1/2}$ is an invertible (unbounded) operator.
	Denote $X_{1/2} = \domain((\omega - A)^{1/2})$ and $X_1 = \domain(\omega - A)$, then
	$M$ generates a strongly continuous group on $X_{1/2} \times X$, whose domain is $X_1 \times X_{1/2}$.
	Furthermore, if $A$ has compact resolvents, then the generalized eigenvectors of $M$ span a dense subspace of $X_{1/2} \times X$.
\end{proposition}
\begin{remark}
	The space $X_{1/2}$ is independent of $\omega$, because $(\omega_1 - A)^{1/2} (\omega_2 - A)^{-1/2}$ is bounded by functional calculus.
	Furthermore, for the operator $A$ defined in \Cref{eq:laplace_op}, $X_{1/2}$ is the concrete Sobolev space $H^1_0(\Omega \setminus \Omega_1)$.
	This is proved by noticing that $(\omega - A)^{-1/2} : L^2 \rightarrow X_{1/2}$ is a Banach space isomorphism, so $X_1 \subseteq L^2$ gets mapped to a dense subset of $X_{1/2}$.
	Since $(\omega - A)^{-1/2} X_1 \subseteq \domain(\omega - A) = X_1$, $X_{1/2}$ is the closure of $X_1$ in the norm $\| \cdot \|_{X_{1/2}}$.
	Finally, the graph norm $\|x\|_{X_{1/2}}^2$ is equivalent to $\|(\omega - A)^{1/2} x\|_{L^2}^2 = \langle (\omega - A)x, x \rangle_{L^2}$, which is also equivalent to $\|x\|_{H^1}^2$.
	Therefore, $X_{1/2}$ is the closure of $\{x \in H^2(\Omega) \mid x = 0 \text{ on } \Omega_1, \hat{n} \cdot \gradient x = 0 \text{ on } \Omega_2 \}$ in $H^1$, which is $H^1_0(\Omega \setminus \Omega_1)$.
\end{remark}
\begin{proof}
	By Lumer-Phillips \cite[Chpt 1, Thm 4.3]{pazy83}, it suffices to show that $\lambda \pm M$ are dissipative and surjective for some $\lambda > 0$.
	Simple calculations show that  
	\begin{align} \labellocal{eq:eigen}
		\left( \begin{bmatrix} \lambda & 0 \\ 0 & \lambda \end{bmatrix} - \begin{bmatrix} 0 & 1 \\ A & -\kappa \end{bmatrix} \right) \begin{bmatrix} x \\ v \end{bmatrix} &= \begin{bmatrix} f \\ g \end{bmatrix}
	\end{align}
	has a solution $(x, v) \in X_1 \times X_{1/2}$ for every $(f, g) \in X_{1/2} \times X$ iff $\lambda(\lambda + \kappa) \not\in \spec(A)$.
	Since $A$ generates a semigroup, $\lambda \pm M$ is surjective as long as $\lambda$ is sufficiently large.
	To show dissipativity of $\lambda \pm M$, note that
	\begin{align*}
		&\phantom{{}={}} \left\langle \begin{bmatrix} 0 & 1 \\ A & -\kappa \end{bmatrix} \begin{bmatrix} x \\ v \end{bmatrix}, \begin{bmatrix} x \\ v \end{bmatrix} \right\rangle_{X_{1/2} \times X} \\
		&= \langle (\omega - A)^{1/2} v, (\omega - A)^{1/2} x \rangle_{X} + \langle Ax, v \rangle_{X} - \kappa \|v\|_{X}^2 \\
		&= \overline{\langle (\omega - A) x, v \rangle_{X}} + \langle Ax, v \rangle_{X} - \kappa \|v\|_{X}^2 \\
		&= 2 \im \langle Ax, v \rangle_{X} + \overline{\langle \omega x, v \rangle_{X}} - \kappa \|v\|_{X}^2.
	\end{align*}
	The first term has zero real part, and the last two terms are bounded by some constant multiple of the norm of $(x, v)$.
	Therefore, $\lambda \pm M$ are dissipative for sufficiently large $\lambda$.
	
	To prove that the eigenvectors of $M$ span a dense subset of $X_{1/2} \times X$, we note that the eigenvectors of $A$ is a complete basis of $X$.
	By \Creflocal{eq:eigen}, for any $\mu \in \spec(A)$ with associated eigenvector $x \in \Ker(\mu - A)$, the vector $(x, \lambda_j x)$ is an eigenvector of $M$ for both roots $\lambda_j$ of $\lambda(\lambda + \kappa) = \mu$.
	In the case $\lambda(\lambda + \kappa) = \mu$ has two distinct roots, the two eigenvectors span $(x, 0)$ and $(0, x)$.
	In the case $\lambda(\lambda + \kappa) = \mu$ has a double root, $(0, x)$ is a generalized eigenvector with eigenvalue $\lambda_j$. 
	Therefore, the generalized eigenvectors of $M$ span a subspace of $X_{1/2} \times X$ containing $(x, 0)$ and $(0, x)$ for all eigenvectors $x$ of $A$.
	Such $x$ form a complete basis of both $X$ and $X_{1/2}$, which implies the desired result.
\end{proof}

For operators $A$ generating strongly continuous groups, there is a general sufficient condition ensuring that $A$ is a Riesz-spectral operator.

\begin{lemma}[\cite{zwart10}] \label{lem:riesz_spectral}
	If $A$ generates a strongly continuous group on a Hilbert space $X$, the eigenvectors of $A$ span a dense subset of $X$, and $\spec(A)$ consists of simple eigenvalues $\{\lambda_k\}_{k\in\NN}$ for which $\inf_{k \neq l} |\lambda_k - \lambda_l| > 0$, then $A$ is a Riesz-spectral operator.
\end{lemma}

Assuming that $A$ is a Riesz-spectral operator, then the equivalence in \Cref{thm:control_schauderbasis} holds for $\system(A, B, C)$.
Furthermore, $A$ has the property that it is exponentially stable if and only if $\spec(A) \subseteq \CCl{-\epsilon}$ for some $\epsilon > 0$ \cite[Thm. 2.3.5]{curtain95}.
Therefore, $A$ has finite unstable part if and only if $\spec(A) \cap \CCge{-\epsilon}$ is finite for some $\epsilon > 0$.

\begin{example}[Wave equation] \label{ex:wave1} \stepcounter{refer}
	Consider the equation
	\begin{align*}
		\left\{
		\begin{aligned}
			(\partial_t^2 + \kappa \partial_t) x(t, \xi) &= \partial_{\xi}^2 x + b x + f(\xi) u(t), \text{ on } \RRge{0} \times [0,1], \\
			\partial_{\xi} x(t, 0) &= 0, \\
			x(t, 1) &= 0, \\
			y(t) &= x(t, 0).
		\end{aligned}
		\right.
	\end{align*}
	It can be modelled as the abstract Cauchy problem on $H^1_0([0,1)) \times L^2([0,1])$:
	\begin{align} \labellocal{eq:acp}
		\begin{aligned}
		\partial_t \begin{bmatrix} x \\ v \end{bmatrix} &= \underbrace{\begin{bmatrix} 0 & 1 \\ A & -\kappa \end{bmatrix}}_{M} \begin{bmatrix} x \\ v \end{bmatrix} + \begin{bmatrix} 0 \\ B \end{bmatrix} u, \\
		y &= \begin{bmatrix} C & 0 \end{bmatrix} \begin{bmatrix} x \\ v \end{bmatrix}
		\end{aligned}
	\end{align}
	where $\domain(A) = \{x \in H^2([0,1]) \mid \partial_{\xi} x(0) = x(1) = 0\}$, $Ax = \partial_{\xi}^2 x + bx$, $Bu = fu$, and $Cx = x(0)$.
	Simple calculations show that $\spec(M)$ consists of eigenvalues
	\begin{align*}
		 \left\{ \frac{-\kappa \pm \sqrt{\kappa^2 - 4\pi^2 k^2 + 4b}}{2} \right\}_{k \in 1/2 + \ZZge{0}},
	\end{align*} 
	where each choice of $k$ and sign is counted with multiplicity one.
	All but a finite number of eigenvalues are simple, hence $M$ is a direct sum of an operator on finite-dimensional space with a Riesz-spectral operator (by \Cref{lem:riesz_spectral}). 
	Therefore, \Creflocal{eq:acp} can be exponentially stabilized iff it is stabilizable, detectable, with finite unstable part.
	By the knowledge of $\spec(M)$, \Creflocal{eq:acp} has finite unstable part iff $\kappa > 0$.
	Stabilizability and detectability are checked by decomposing the system along its generalized eigenspaces.
	For each $k \in 1/2 + \ZZge{0}$, the eigenspaces to eigenvalues $\{\lambda \mid \lambda(\lambda + \kappa) = b - \pi^2 k^2\}$ span the subspace
	\begin{align*}
		\Span\left\{ \begin{bmatrix} \cos(\pi k \xi) & 0 \end{bmatrix}^\top, \begin{bmatrix} 0 & \cos(\pi k \xi) \end{bmatrix}^\top \right\},
	\end{align*}
	and the projection of \Creflocal{eq:acp} to this subspace is
	\begin{align*}
		\partial_t \begin{bmatrix} x \\ v \end{bmatrix} &= \begin{bmatrix} 0 & 1 \\ b - \pi^2 k^2 & -\kappa \end{bmatrix} \begin{bmatrix} x \\ v \end{bmatrix} + \begin{bmatrix} 0 \\ \frac{\langle \cos(\pi k \xi), f \rangle_{L^2}}{\langle \cos(\pi k \xi), \cos(\pi k \xi) \rangle_{L^2}} \end{bmatrix} u, \\
		y &= \begin{bmatrix} 1 & 0 \end{bmatrix} \begin{bmatrix} x \\ v \end{bmatrix}.
	\end{align*}
	\Creflocal{eq:acp} is stabilizable and detectable iff every projection is stabilizable and detectable.
	One can verify that the projection is always detectable, and it is stabilizable if either the system stable or $\langle \cos(\pi k \xi), f \rangle_{L^2} \neq 0$.
	In summary, \Creflocal{eq:acp} can be exponentially stabilized iff $\kappa > 0$ and $\langle \cos(\pi k \xi), f \rangle_{L^2} \neq 0$ for all $k \in 1/2 + \ZZge{0}$ such that $\pi^2 k^2 \leq b$.
\end{example}

\section{Conclusion}

This paper investigated the exponential stabilizability of an infinite-dimensional system by a finite-dimensional controller.
We identified that the problem is only solvable if the system is stabilizable, detectable, and with finite unstable part.
Furthermore, exponential stabilization is solvable if the system is also quasi-finite.
How much our sufficient conditions are stronger than the necessary conditions remain unclear. 
We have identified some conditions guaranteeing quasi-finiteness, a precise characterization of quasi-finiteness is a direction of future research.

\bibliographystyle{IEEEtran} 
\bibliography{PDEControl}

\begin{thebibliography}{10}
\providecommand{\url}[1]{#1}
\csname url@samestyle\endcsname
\providecommand{\newblock}{\relax}
\providecommand{\bibinfo}[2]{#2}
\providecommand{\BIBentrySTDinterwordspacing}{\spaceskip=0pt\relax}
\providecommand{\BIBentryALTinterwordstretchfactor}{4}
\providecommand{\BIBentryALTinterwordspacing}{\spaceskip=\fontdimen2\font plus
\BIBentryALTinterwordstretchfactor\fontdimen3\font minus
  \fontdimen4\font\relax}
\providecommand{\BIBforeignlanguage}[2]{{%
\expandafter\ifx\csname l@#1\endcsname\relax
\typeout{** WARNING: IEEEtran.bst: No hyphenation pattern has been}%
\typeout{** loaded for the language `#1'. Using the pattern for}%
\typeout{** the default language instead.}%
\else
\language=\csname l@#1\endcsname
\fi
#2}}
\providecommand{\BIBdecl}{\relax}
\BIBdecl

\bibitem{curtain95}
R.~F. Curtain and H.~Zwart, \emph{An Introduction to Infinite-Dimensional
  Linear Systems Theory}.\hskip 1em plus 0.5em minus 0.4em\relax New York, NY,
  USA: Springer, 1995.

\bibitem{krstic08}
M.~Krstic and A.~Smyshlyaev, \emph{Boundary control of PDEs: A course on
  backstepping designs}.\hskip 1em plus 0.5em minus 0.4em\relax SIAM, 2008.

\bibitem{curtain82}
R.~Curtain, ``Finite-dimensional compensator design for parabolic distributed
  systems with point sensors and boundary input,'' \emph{IEEE Transactions on
  Automatic Control}, vol.~27, no.~1, pp. 98--104, 1982.

\bibitem{balas88}
M.~Balas, ``Finite-dimensional controllers for linear distributed parameter
  systems: Exponential stability using residual mode filters,'' \emph{Journal
  of Mathematical Analysis and Applications}, vol. 133, no.~2, pp. 283--296,
  1988.

\bibitem{fattorini1971exact}
H.~O. Fattorini and D.~L. Russell, ``Exact controllability theorems for linear
  parabolic equations in one space dimension,'' \emph{Archive for Rational
  Mechanics and Analysis}, vol.~43, no.~4, pp. 272--292, 1971.

\bibitem{prieur19}
C.~Prieur and E.~Tr{\'e}lat, ``Feedback stabilization of a 1-d linear
  reaction-diffusion equation with delay boundary control,'' \emph{IEEE
  Transactions on Automatic Control}, vol.~64, no.~4, pp. 1415--1425, 2019.

\bibitem{mironchenko2020local}
A.~Mironchenko, C.~Prieur, and F.~Wirth, ``Local stabilization of an unstable
  parabolic equation via saturated controls,'' \emph{IEEE Transactions on
  Automatic Control}, vol.~66, no.~5, pp. 2162--2176, 2020.

\bibitem{russell1978controllability}
D.~L. Russell, ``Controllability and stabilizability theory for linear partial
  differential equations: recent progress and open questions,'' \emph{Siam
  Review}, vol.~20, no.~4, pp. 639--739, 1978.

\bibitem{coron2004global}
J.-M. Coron and E.~Tr{\'e}lat, ``Global steady-state controllability of
  one-dimensional semilinear heat equations,'' \emph{SIAM journal on control
  and optimization}, vol.~43, no.~2, pp. 549--569, 2004.

\bibitem{coron2006global}
------, ``Global steady-state stabilization and controllability of 1d
  semilinear wave equations,'' \emph{Communications in Contemporary
  Mathematics}, vol.~8, no.~04, pp. 535--567, 2006.

\bibitem{katz20}
R.~Katz and E.~Fridman, ``Constructive method for finite-dimensional
  observer-based control of 1-{D} parabolic {PDE}s,'' \emph{Automatica}, vol.
  122, p. 109285, 2020.

\bibitem{lhachemi22}
H.~Lhachemi and C.~Prieur, ``Finite-dimensional observer-based boundary
  stabilization of reaction-diffusion equations with a either dirichlet or
  neumann boundary measurement,'' \emph{Automatica}, vol. 135, p. 109955, 2022.

\bibitem{katz2020boundary}
R.~Katz, E.~Fridman, and A.~Selivanov, ``Boundary delayed observer-controller
  design for reaction--diffusion systems,'' \emph{IEEE Transactions on
  Automatic Control}, vol.~66, no.~1, pp. 275--282, 2020.

\bibitem{lhachemi2022boundary}
H.~Lhachemi and C.~Prieur, ``Boundary output feedback stabilisation of a class
  of reaction--diffusion pdes with delayed boundary measurement,''
  \emph{International Journal of Control}, pp. 1--11, 2022.

\bibitem{lhachemi2023output}
H.~Lhachemi and R.~Shorten, ``Output feedback stabilization of an
  {ODE}-reaction--diffusion pde cascade with a long interconnection delay,''
  \emph{Automatica}, vol. 147, p. 110704, 2023.

\bibitem{lhachemi2021finite}
H.~Lhachemi and C.~Prieur, ``Finite-dimensional observer-based pi regulation
  control of a reaction--diffusion equation,'' \emph{IEEE Transactions on
  Automatic Control}, vol.~67, no.~11, pp. 6143--6150, 2021.

\bibitem{lhachemi2022proportional}
H.~Lhachemi, C.~Prieur, and E.~Tr{\'e}lat, ``Proportional integral regulation
  control of a one-dimensional semilinear wave equation,'' \emph{SIAM Journal
  on Control and Optimization}, vol.~60, no.~1, pp. 1--21, 2022.

\bibitem{lhachemi2022local}
H.~Lhachemi and C.~Prieur, ``Local output feedback stabilization of a
  reaction-diffusion equation with saturated actuation,'' \emph{IEEE
  Transactions on Automatic Control}, 2022.

\bibitem{dus21}
M.~Dus, F.~Ferrante, and C.~Prieur, ``Spectral stabilization of linear
  transport equations with boundary and in-domain couplings,'' \emph{Comptes
  Rendus Math{\'e}matique}, 2021.

\bibitem{shreim2022input}
S.~Shreim, F.~Ferrante, and C.~Prieur, ``Input-output stability of a reaction
  diffusion equation with in-domain disturbances,'' in \emph{2022 IEEE 61st
  Conference on Decision and Control (CDC)}.\hskip 1em plus 0.5em minus
  0.4em\relax IEEE, 2022, pp. 403--408.

\bibitem{dunford71}
N.~Dunford and J.~T. Schwartz, \emph{Linear Operators, Part 3: Spectral
  Operators}.\hskip 1em plus 0.5em minus 0.4em\relax Wiley-Interscience, 1971.

\bibitem{pazy83}
A.~Pazy, \emph{Semigroups of Linear Operators and Applications to Partial
  Differential Equations}.\hskip 1em plus 0.5em minus 0.4em\relax New York, NY,
  USA: Springer, 1983.

\bibitem{kato80}
T.~Kato, \emph{Perturbation Theory for Linear Operators}, 2nd~ed.\hskip 1em
  plus 0.5em minus 0.4em\relax Springer-Verlag, 1980.

\bibitem{zwart10}
H.~Zwart, ``Riesz basis for strongly continuous groups,'' \emph{Journal of
  Differential Equations}, vol. 249, no.~10, pp. 2397--2408, 2010.

\end{thebibliography}

\end{document}